\documentclass[reqno,a4paper,10pt]{amsart}
\usepackage[english]{babel}
\usepackage{latexsym}
\usepackage{amsmath,amssymb, amsfonts,amsthm}
\usepackage{verbatim}
\usepackage{color}
\usepackage[dvipsnames]{xcolor}
\usepackage{fancyhdr} 
\usepackage{lastpage} 
\usepackage[T1]{fontenc}
\usepackage{calligra}
\usepackage{dsfont}

\newtheorem{theorem}{Theorem}[section]
\newtheorem{lemma}[theorem]{Lemma}

\newtheorem{corollary}[theorem]{Corollary}

\theoremstyle{remark}
\newtheorem{remark}[theorem]{Remark}

\numberwithin{equation}{section}

\newcommand{\R}{{\mathbb R}}
\newcommand{\N}{{\mathbb N}}
\newcommand{\Z}{{\mathbb Z}}

\newcommand{\nH}{h} 
\newcommand{\supp}{\mathop{\bf supp}}

\DeclareMathOperator*{\Rn}{(0,\infty)^{n}}

\begin{document}

\title{A Liouville theorem for some Bessel generalized operators.}

\author[V. Galli]{Vanesa Galli}
\address{Vanesa Galli,
	Departamento de Matem\'atica, Facultad de Ciencias Exactas y
	Naturales,  Universidad Nacional de Mar del Plata, Funes 3350, 7600
	Mar del Plata, Argentina.}
\email{vggalli@mdp.edu.ar}

\author[S. Molina]{Sandra Molina}
\address{Sandra Molina, Departamento de Matem\'atica, Facultad de Ciencias Exactas y
	Naturales,  Universidad Nacional de Mar del Plata, Funes 3350, 7600
	Mar del Plata, Argentina}
\email{smolina@mdp.edu.ar }

\author[A. Quintero]{Alejandro Quintero}
\address{Alejandro Quintero,
	Departamento de Matem\'atica, Facultad de Ciencias Exactas y
	Naturales,  Universidad Nacional de Mar del Plata, Funes 3350, 7600
	Mar del Plata, Argentina.}
\email{aquinter@mdp.edu.ar}

\subjclass[2010]{ Primary 46F12, 47F05, 42A38, 44A05}
\keywords{Liouville theorem; Bessel Operator; Hankel transform.}

\date{}

\dedicatory{}

\begin{abstract}In this paper we establish a Liouville theorem in $\mathcal{H'}_{\mu}$ for a wider class of operators in $\Rn$ that generalizes the $n$-dimensional Bessel operator. We will present two different proofs, based in two representation theorems for certain distributions  "supported in zero". 
\end{abstract}
\maketitle

\section{Introduction}

Liouville type theorems have been studied in many works under different contexts. In analytic theory, Liouville theorems stated that a bounded entire function reduces to a constant. A  first version of Liouville theorem in distributional theory is due to L. Schwartz \cite{Schwartz2}, and assert that any bounded harmonic function in $\R^{n}$ is a constant.

Currently, this result has been generalized in many directions. A well known generalization states that:

\medskip

\textit{Let $L=\underset{|\alpha|\leq m}{\sum}a_{\alpha}D^{\alpha}$ be a linear differential operator with constant coefficients such that $L=\underset{|\alpha|\leq m}{\sum}a_{\alpha}(2\pi i\xi)^{\alpha}\neq 0$ for all $\xi\in\R^{n}-\{0\}$. If a tempered distribution $u$, solves $Lu=0$, then $u$ is a polynomial function. In particular, if $u$ is bounded then it reduces to a constant. }



\medskip

In this work, we established a Liouville type theorem for a large class of operators in $\Rn$, that are lineal combinations of operators 

\begin{equation}\label{OpS}
S^{k}=S_{\mu_{1}}^{k_{1}}\circ\ldots\circ S_{\mu_{n}}^{k_{n}}
\end{equation}

\noindent where $k$ is a multi-index, $k=(k_{1}, \ldots, k_{n})$, $\mu_{i}\in\R$, $\mu_{i}\geq -1/2$ and 

\begin{equation}\label{Bessel}
S_{\mu_{i}}=\frac{\partial^{2}}{\partial x_{i}^{2}}-\frac{4\mu_{i}-1}{4x_{i}^2}.
\end{equation}

The operators given by linear combination of (\ref{OpS})  contain as a  particular case the $n$-dimensional operator defined in \cite{Molina2} and given by:

\begin{equation}\label{nB}
S_{\mu}=\Delta-\sum_{i=1}^{n}\frac{4\mu_{i}^{2}-1}{4x_{i}^{2}}
\end{equation}

\noindent where $\mu=(\mu_{1},\ldots,\mu_{n})$ and $S_{\mu}$ is a $n$-dimensional version of the well know Bessel operator

\begin{equation}\label{1B}
S_{\alpha}=\frac{d^{2}}{d x^{2}}-\frac{4\alpha^{2}-1}{4x^2}.
\end{equation}

\noindent This operators were introduced in relation to the Hankel transform given by 

\begin{equation}\label{1Hankel}
h_{\alpha}f(y)=\int_{0}^{\infty}f(x)\sqrt{xy}J_{\alpha}(xy)dx
\end{equation}
\noindent with $\alpha\geq -1/2$, for 1-dimensional case and the $n$-dimensional case

\begin{equation}\label{nHZ}
(h_{\mu}\phi)(y)=\int_{\Rn}\phi(x_{1},\ldots,x_{n})\prod_{i=1}^{n}\{\sqrt{x_{i}y_{i}}J_{\mu_{i}}(x_{i}y_{i})\}\;dx_{1}\ldots dx_{n}
\end{equation}

\noindent with $\mu=(\mu_{1},\ldots,\mu_{n})$, $\mu_{i}\geq -1/2$, $i=1,\ldots,n$. And $J_{\nu}$ represents the Bessel functions of the first kind and order $\nu$.

\medskip

Bessel operators  (\ref{nB}) and (\ref{1B}) and Hankel Transforms  (\ref{1Hankel}) and (\ref{nHZ})  were studied on Zemanian spaces $\mathcal{H}_{\mu}$ and $\mathcal{H'}_{\mu}$ in \cite{Molina2}, \cite{Molina3} and \cite{Zemanian2}.

\medskip

The space $\mathcal{H}_\mu$ is a space  of functions $\phi\in C^{\infty}(\Rn)$ such that for all $m\in\N_{0}, k\in\N_{0}^{n}$ verifies

\begin{equation}\label{S1}
\gamma_{m,k}^{\mu}(\phi)=\sup_{x\in\Rn}|(1+\|x\|^2)^{m}T^{k}\{x^{-\mu-1/2}\phi(x)\}|<\infty
\end{equation}

\noindent where $-\mu-1/2=(-\mu_1-1/2,\ldots,-\mu_n-1/2)$ and  the operators $T^{k}$ are given by

$$T^{k}=T_{n}^{k_n}\circ T_{n-1}^{k_{n-1}}\circ\ldots\circ T_{1}^{k_1},$$

\noindent where $T_{i}=x_{i}^{-1}\frac{\partial}{\partial x_{i}}$. Thus $\mathcal{H}_{\mu}$ is Fr\`{e}chet space. The dual space of $\mathcal{H}_\mu$ is denoted by $\mathcal{H'}_\mu$.

\medskip

In \cite{Molina2} the authors proved that $S_{\mu_{i}}$ is continuous from $\mathcal{H}_{\mu}$ into itself for all $i=1,\ldots,n$  and self-adjoints lineal mappings.   Which also implies that the operator $S^{k}=S_{\mu_{n}}^{k_{n}}\ldots S_{\mu_{1}}^{k_{1}}$ is continuous from $\mathcal{H}_{\mu}$ into itself. Then, since they are self-adjoints  the generalized operators  can be extended to $\mathcal{H'}_{\mu}$ by
		
\begin{equation}
(S_{\mu_{i}}f, \phi)=(f, S_{\mu_{i}}\phi)\quad\text{and} \quad
(S^{k}f, \phi)=(f, S^{k}\phi), \quad f\in\mathcal{H'}_{\mu}\quad \phi\in \mathcal{H}_{\mu}.\\
\end{equation}

\medskip

\noindent The generalized Hankel transformation $h_{\mu}f$ of $f\in \mathcal{H'}_\mu$ is defined   by

\begin{equation*}
(h_{\mu}f, \phi)=(f, h_{\mu}\phi),\qquad f\in\mathcal{H'}_{\mu}, \qquad \phi\in \mathcal{H}_\mu.
\end{equation*}

\noindent for $\mu\in[-1/2,\infty)^{n}$.
\noindent Then $h_{\mu}$ is an automorphism onto $\mathcal{H}_{\mu}$ and $\mathcal{H'}_{\mu}$ and $h_{\mu}=(h_{\mu})^{-1}$.  

\medskip

The Hankel transform and Bessel operator are related by 
\begin{equation*}
\nH_{\mu}(S_{\mu})=-\|y^{2}\|\nH_{\mu},
\end{equation*}

\noindent in $\mathcal{H}_\mu$ and $\mathcal{H'}_\mu$. 

\medskip

Now we shall describe the main result of this work.

\begin{theorem}\label{GeneralLiouville}
Let  $P[x]$ be a polynomial in $n$-variables such that  $\underset{|\alpha|\leq N}{\sum} a_{\alpha}x^{\alpha}\neq0$ for all  $x\in\R^{n}-\{0\}$ and all its coefficients have the same sign. Let $L$ be the operator
$L=\underset{|\alpha|\leq N}{\sum}(-1)^{|\alpha|}a_{\alpha}S^{\alpha}$.
If $f\in\mathcal{H'}_{\mu}$ and 
\begin{equation}\label{GL}
Lf=0
\end{equation} 
then there exist a polynomial in $n$-variables $Q$ such that $f(x)=x^{\mu+1/2}Q[x_{1}^{2},\ldots, x_{n}^{2}]$.
\end{theorem}

\begin{corollary}
If $f$ is a classical solution of (\ref{GL}) of slow growth then there exist a polynomial in $n$-variables $Q$ such that $f(x)=x^{\mu+1/2}Q[x_{1}^{2},\ldots, x_{n}^{2}]$.. In particular, if $f$ is bounded then $f$ is a constant.
\end{corollary}

\begin{remark}
	The cases $\mu=(\mu_{1},\ldots,\mu_{n})=(1/2,\ldots,1/2)$ or $(-1/2,\ldots,-1/2)$ produce in (\ref{nB}) the Laplacian operator in $\Rn$. 
\end{remark}

This paper is organized as follows. In Section 2, we present some notational conventions that will allow us to simplify the
presentation of our results. In Section 3 we propose a characterization of a certain family of functions on the  multiplier space $\mathcal{O}$ of the $n$-dimensional space $\mathcal{H}_{\mu}$ that extends the result proved by Zemanian in \cite{Zemanian2}. In Section 4 and 5 we give two different  proofs of Theorem \ref{GeneralLiouville}.

Throughout this paper, C always will denote a suitable positive constant that may change in each occurrence.

\section{Preliminaries and notations}
In this section we summarize without proof the relevant material on Hankel transforms and the Zemanian spaces studied  in \cite{Molina1}, \cite{Molina2} and \cite{Molina3}.

We now present some notational conventions that will allow us to simplify the presentation of our results. We denote by $x=(x_1,\ldots,x_n)$ and $y=(y_1,\ldots,y_n)$ elements of $\Rn$ or $\R^{n}$. Let $\N$ be the set $\{1,2,3,\ldots\}$ and $\N_{0}=\N\cup\{0\}$, $\|x\|=(x_{1}^{2}+\ldots + x_{n}^{2})^{\frac{1}{2}}$. The notations $x<y$ and $x\leq y$ mean, respectively,  $x_i<y_i$ and  $x_i\leq y_i$ for $i=1,\ldots,n$. Moreover,  $x=a$ for $x\in\R^{n}$, $a\in\R$ means $x_1=x_2=\ldots=x_n=a$, $x^{m}=x_{1}^{m_1}\ldots x_{n}^{m_n}$ and $e_j$ for $j=1,\ldots,n$, denotes the members of the canonical basis of $\R^{n}$. An element $k=(k_1,\ldots,k_n)\in\N_{0}^{n}=\N_{0}\times\N_{0}\times\ldots\times\N_{0}$ is called multi-index. For $k,m$ multi-index we set $|k|=k_1+\ldots+k_n$ the length of the multi-index. 

Also we will note                                            $$k!=k_1!\ldots k_n!, \qquad \binom{k}{m}=\binom{k_1}{m_1}\ldots  \binom{k_n}{m_n} \qquad \text{for} \: k,m\in\N_{0}^{n}.$$

\begin{remark}
Let $k$ be a multi-index, the following equality is valid
\begin{equation}\label{Leibniz}
  T^{k}\{\theta.\varphi\}=\sum_{j=0}^{k}\binom{k}{j}T^{k-j}\theta. T^{j}\varphi,
\end{equation}

\noindent where "$.$" denote the usual product of functions, $\binom{k}{j}$ and $\sum_{j=0}^{k}$ must be interpreted as in the previous section for $j=0=(0,\ldots,0)$.
\end{remark}

\begin{remark}
If $e_{i}$ is an element of the canonical base of $\R^{n}$, since
$S^{e_{i}}=S_{\mu_{n}}^{0}\ldots\circ S_{\mu_{i}}^{1}\ldots S_{\mu_{1}}^{0}=S_{\mu_{i}}$, then
$$\sum_{i=1}^{n}S^{e_{i}}=\sum_{i=1}^{n}S_{\mu_{i}}=S_{\mu}.$$
\end{remark}

In \cite{Betancor1} was defined  the generalized function $\delta_{\alpha}$, as 

\begin{eqnarray}\label{A}
(\delta_{\alpha},\phi)=C_{\alpha}\lim\limits_{x\to 0^{+}}x^{-\alpha-1/2}\phi(x),
\end{eqnarray}

\noindent where $C_{\alpha}=2^{\alpha}\Gamma(\alpha+1)$. The distribution given by (\ref{A}) can be extended in the same way to the $n$-dimensional case. Moreover we can consider the following distribution 

\begin{equation}\label{T^{k}delta}
(T^{k}\delta_{\mu}, \phi)=C_{\mu} \lim\limits_{\underset{x_{i}>0}{x\to 0}}
T^{k}\{x^{-\mu-1/2}\phi(x)\}
\end{equation}

\noindent where $k$ is a multi-index,  $\mu\in\R^{n}$  and  $C_{\mu}$ is a constant depending on $\mu$ given by $C_{\mu}=\prod_{i=1}^{n} 2^{\mu_{i}}\Gamma(\mu_{i}+1)$. The generalized function (\ref{T^{k}delta}) is well defined as it can be seen in the proof of Lemma \ref{Taylor}. 

Let $\phi\in\mathcal{H}_{\mu}$, since

$$|(T^{k}\delta_{\mu}, \phi)|= |C_{\mu}\lim\limits_{\underset{x_{i}>0}{x\to 0}}
T^{k}\{x^{-\mu-1/2}\phi(x)\}|\leq
C_{\mu}\sup_{x\in\Rn}|T^{k}\{x^{-\mu-1/2}\phi(x)\}|= C_{\mu}\gamma_{0,k}^{\mu}(\phi),$$

\noindent so $T^{k}\delta_{\mu}$ lies in $\mathcal{H'}_{\mu}$. Moreover,
\begin{equation}\label{HankelT^{k}}
h_{\mu}T^{k}\delta_{\mu}=C_{k}^{\mu}\:t^{\mu+2k+1/2}\quad\text{in} \quad \mathcal{H'}_{\mu},
\end{equation}

\noindent where $C_{k}^{\mu}=(-1)^{|k|}\frac{C_{\mu}}{C_{\mu+k}}$. Indeed, since the well known formula $\frac{d}{dz}(z^{-\alpha}J_{\alpha})=-z^{-\alpha}J_{\alpha+1}$ is valid for $\alpha\neq -1,-2,\ldots$,  if we consider $k=e_{j}$, then 

\begin{align*}
&(h_{\mu}T_{j}\delta_{\mu}, \phi)=
(T_{j}\delta_{\mu}, h_{\mu}\phi)=
C_{\mu}\lim\limits_{\underset{x_{i}>0}{x\to 0}}
T_{j}\{x^{-\mu-1/2}h_{\mu}\phi(x)\}=\\
&=C_{\mu}\lim\limits_{\underset{x_{i}>0}{x\to 0}}
x_{j}^{-1}\partial/\partial x_{j}\left\{x^{-\mu-1/2}
\int_{\Rn}\phi(t_{1},\ldots,t_{n})
\prod_{i=1}^{n}\{\sqrt{x_{i}t_{i}}J_{\mu_{i}}(x_{i}t_{i})\}
dt_{1}\ldots dt_{n}\right\}\\
&=C_{\mu}\lim\limits_{\underset{x_{i}>0}{x\to 0}}
x_{j}^{-1}\partial/\partial x_{j}\left\{
\int_{\Rn}t^{\mu+1/2}\phi(t_{1},\ldots,t_{n})
\prod_{i=1}^{n}\{(x_{i}t_{i})^{-\mu_{i}}J_{\mu_{i}}(x_{i}t_{i})\}
dt_{1}\ldots dt_{n}\right\}\\
&=C_{\mu}\lim\limits_{\underset{x_{i}>0}{x\to 0}}
x_{j}^{-1}\left\{
\int_{\Rn}t^{\mu+1/2}\phi(t_{1},\ldots,t_{n})\;
\partial/\partial x_{j} \left\{
\prod_{i=1}^{n}
\{(x_{i}t_{i})^{-\mu_{i}}J_{\mu_{i}}(x_{i}t_{i})\}\right\}
dt_{1}\ldots dt_{n}\right\}\\
&=-C_{\mu}\lim\limits_{\underset{x_{i}>0}{x\to 0}}
\int_{\Rn}t^{\mu+2e_{j}+1/2}\phi(t_{1},\ldots,t_{n})\;
\left[(x_{j}t_{j})^{-(\mu_{j}+1)}J_{\mu_{j}+1}(x_{j}t_{j})\right]
\prod_{\underset{i\neq j}{i=1}}^{n}
\{(x_{i}t_{i})^{-\mu_{i}}J_{\mu_{i}}(x_{i}t_{i})\}
dt_{1}\ldots dt_{n}\\
&=-C_{\mu}\{C_{\mu_{j}+1}\prod_{\underset{i\neq j}{i=1}}^{n}C_{\mu_{i}}\}^{-1}
\int_{\Rn}t^{\mu+2e_{j}+1/2}\phi(t_{1},\ldots,t_{n})\;
dt_{1}\ldots dt_{n}\\
&=
\left(-\frac{C_{\mu_{j}}}{C_{\mu_{j}+1}}t^{\mu+2e_{j}+1/2},\phi\right)\\
\end{align*}

\noindent Therefore the assertion is true for $k=e_{j}$. The general case follows in a similar way. Indeed, let $r\in\N_{0}$ and let us observe that \begin{align}\label{eq1}
&T_{j}^{r}
\left\{
t^{\mu+1/2}
\prod_{i=1}^{n}\{(x_{i}t_{i})^{-\mu}J_{\mu_{i}}(x_{i}t_{i})\}
\right\}=\nonumber\\
&=(x_{j}^{-1}\partial/\partial x_{j})^{r}
\left\{
t^{\mu+1/2}
\prod_{i=1}^{n}\{(x_{i}t_{i})^{-\mu}J_{\mu_{i}}(x_{i}t_{i})\}
\right\}\nonumber\\
&=(-1)^{r}t^{\mu+2r\,e_{j}+1/2}
(x_{j}t_{j})^{-(\mu_{j}+r)}J_{\mu_{j}+r}(x_{j}t_{j})
\prod_{\underset{i\neq j}{i=1}}^{n}
\{(x_{i}t_{i})^{-\mu_{i}}J_{\mu_{i}}(x_{i}t_{i})\},
\end{align} 
\noindent then (\ref{eq1}) yields 
\begin{align*}
&(h_{\mu}T_{j}^{r}\delta_{\mu},\phi)=
(T_{j}^{r}\delta_{\mu},h_{\mu}\phi)=
C_{\mu}\lim\limits_{\underset{x_{i}>0}{x\to 0}} T_{j}^{r}\{x^{-\mu-1/2}h_{\mu}\phi(x)\}=\\
&=C_{\mu}\lim\limits_{\underset{x_{i}>0}{x\to 0}}
(x_{j}^{-1}\partial/\partial x_{j})^{r}
\left\{
x^{-\mu-1/2}\int_{\Rn}\phi(t_{1},\ldots,t_{n})
\prod_{i=1}^{n}
\{\sqrt{x_{i}t_{i}}J_{\mu_{i}}(x_{i}t_{i})\}
dt_{1}\ldots dt_{n}
\right\}\\
&=C_{\mu}\lim\limits_{\underset{x_{i}>0}{x\to 0}}
(x_{j}^{-1}\partial/\partial x_{j})^{r}
\left\{
\int_{\Rn}t^{\mu+1/2}\phi(t_{1},\ldots,t_{n})
\prod_{i=1}^{n}
\{(x_{i}t_{i})^{-\mu_{i}}J_{\mu_{i}}(x_{i}t_{i})\}
dt_{1}\ldots dt_{n}
\right\}\\
&=(-1)^{r}C_{\mu}\lim\limits_{\underset{x_{i}>0}{x\to 0}}
\int_{\Rn}
t^{\mu+2r\,e_{j}+1/2}\phi(t_{1},\ldots,t_{n})
(x_{j}t_{j})^{-(\mu_{j}+r)}J_{\mu_{j}+r}(x_{j}t_{j})
\prod_{\underset{i\neq j}{i=1}}^{n}
\{(x_{i}t_{i})^{-\mu_{i}}J_{\mu_{i}}(x_{i}t_{i})\}
dt_{1}\ldots dt_{n}\\
&=(-1)^{r}C_{\mu}\{C_{\mu_{j}+r}\prod_{\underset{i\neq j}{i=1}}^{n}C_{\mu_{i}}\}^{-1}
\int_{\Rn}t^{\mu+2re_{j}+1/2}\phi(t_{1},\ldots,t_{n})\;
dt_{1}\ldots dt_{n}\\
&=
\left((-1)^{r}\frac{C_{\mu_{j}}}{C_{\mu_{j}+r}}t^{\mu+2re_{j}+1/2},\phi\right).\\
\end{align*}

\noindent For the general case, if we compute for  $j\neq k\in\{1,\ldots,n\}$ and $r,m\in\N_{0}$ we obtain that

\begin{align*}
&(h_{\mu}T_{j}^{r}T_{k}^{m}\delta_{\mu},\phi)=
(T_{j}^{r}T_{k}^{m}\delta_{\mu},h_{\mu}\phi)=
((-1)^{r+m}\frac{C_{\mu_{j}}C_{\mu_{k}}}{C_{\mu_{j}+r}C_{\mu_{k}+m}}
t^{\mu+2re_{j}+2me_{k}+1/2},\phi)
\end{align*}
\noindent and the result follows.

\section{Some results about Taylor's expansions and a  special family of multipliers in $\mathcal{H}_{\mu}$}

In this section we extend the characterization obtained by Zemanian in \cite{Zemanian2} related to Taylor's expansions of functions in $\mathcal{H}_{\mu}$. Moreover we give a result which improve the Lemma 3.2 in \cite{Molina1}.

\begin{lemma}\label{Taylor}
Let $\mu\in\R^{n}$. Then $\phi$ is a member of $\mathcal{H}_{\mu}$ if and only if it satisfies the following three conditions:
	
	\begin{enumerate}
		\item[(i)] $\phi(x)$ is a smooth complex valued function on $\Rn$.
		\item[(ii)] For each $r\in\N_{0}$
		\begin{equation}\label{taylor1}
		x^{-\mu-1/2}\phi(x)=a_{0}+
		\underset{|k_{1}|=1}{\sum}a_{2k_{1}}x^{2k_{1}}+
		\underset{|k_{2}|=2}{\sum}a_{2k_{2}}x^{2k_{2}}+
		\ldots+
		\underset{|k_{r}|=r}{\sum}a_{2k_{r}}x^{2k_{r}}+
		R_{2r}(x),
		\end{equation}
		
		\noindent where 
		
		\begin{equation}\label{taylor2}
		a_{2k_{r}}=\frac{1}{2^{r}k_{r}!}
		\lim_{\underset{x_{i}>0}{x\to 0}}T^{k_{r}}\{x^{-\mu-1/2}\phi(x)\},
		\end{equation}
		
		\noindent and the remainder term $R_{2r}(x)$ satisfies 		
		\begin{equation}\label{taylor3}
		T^{k}R_{2r}(x)=o(1) \qquad \underset{x_{i}>0}{x\to 0}
		\end{equation}
		\noindent for $k$ multi-index such that $|k|=r$.
		
		\item[(iii)] For each multi-index $k_{r}$, $D^{k_{r}}\phi(x)$ is of rapid descent as $|x|\to\infty$.  
	\end{enumerate}
\end{lemma}

\begin{proof}
	
	Since $\phi(x)\in\mathcal{H}_{\mu}$ condition $(i)$ is satisfied by definition. For  $k$ a multi-index let us consider the smooth function in $\Rn$ given by
	\begin{equation}\label{taylor4}
	\psi(x)=\psi(x_{1},\ldots,x_{n})=T^{k}\{x^{-\mu-1/2}\phi(x)\}.
	\end{equation}
	Let us see that the coefficients given by  (\ref{taylor2}) are well defined, $i.e.$,
	
	\begin{equation}\label{taylor7}
	\lim\limits_{\underset{x_{i}>0}{(x_{1},\ldots,x_{n})\to (0,\ldots,0)}} 
	\psi(x_{1},\ldots,x_{n})<\infty.
	\end{equation}

	Since \begin{equation}\label{taylor5}
	\left|\frac{\partial^{n}}{\partial x_{n}\ldots\partial x_{1}}\psi(x)\right|
	\leq M|x_{1}\ldots x_{n}|,
	\end{equation} 
	
\noindent if $(a_{1},\ldots,a_{n})$ in $[0,\infty)^{n}$ such that there exist $1\leq j\leq n$ and $a_{j}=0$ then

$$\lim\limits_{\underset{x_{i}>0}{(x_{1},\ldots,x_{n})\to (a_{1},\ldots,a_{n})}} \frac{\partial^{n}}{\partial x_{n}\ldots\partial x_{1}}\psi(x)=0$$

\noindent then $\frac{\partial^{n}}{\partial x_{n}\ldots\partial x_{1}}\psi(x)$ is $C^{\infty}$ in $\Rn$, continuous in $[0,\infty)^{n}$ and consequently integrable in $[0,1]^{n}$.

\noindent Moreover,

\begin{equation}\label{eq10}
\int_{1}^{x_{1}}\ldots \int_{1}^{x_{n}}
\frac{\partial^{n}}{\partial y_{n}\ldots \partial y_{1}}
\psi(y_{1},\ldots,y_{n})
dy_{n}\ldots dy_{1}
=\psi(x_{1},\ldots,x_{n})+\sum_{\lambda}a_{\lambda}\psi(b_{\lambda})
\end{equation}
\noindent with $a_{\lambda}=1$ or $-1$ and $b_{\lambda}=(b_{\lambda_1},\ldots,b_{\lambda_n})$ with $b_{\lambda_i}=x_{i}$ or $b_{\lambda_i}=1$.

\noindent Let us see now 

$$
\lim\limits_{\underset{x_{i}>0}{(x_{1},\ldots,x_{n})\to (0,\ldots,0)}}\psi(b_{\lambda})<\infty.
$$

\noindent First, let us consider $b_{\lambda}$ such that  $b_{\lambda_{j}}=1$ if $j\neq i$ and $b_{\lambda_{i}}=x_{i}$. Since 

$$\left|\frac{\partial}{\partial y_{i}} 
\psi(y_{1},\ldots,y_{n})
\right|\leq M|y_{i}|,$$

$$\lim\limits_{x_{i}\to 0}
\int_{1}^{x_{i}}
\frac{\partial}{\partial y_{i}} 
\psi(1,\ldots,y_{i},\ldots,1)dy_{i}<\infty.
$$

\noindent So, $\lim\limits_{x\to 0}
\psi(1,\ldots,x_{i},\ldots,1)<\infty$.

\noindent Now let us see that $\lim\limits_{x\to 0}
\psi(1,\ldots,x_{i},1,\ldots,1,x_{j}\ldots,1)<\infty$. In fact 

$$
\left|\frac{\partial}{\partial x_{i}\partial x_{j}}
\psi(1,\ldots,1,x_{i},1,\ldots,1,x_{j},\ldots,1)
\right|
\leq M|x_{i}x_{j}|.
$$

Then $\frac{\partial}{\partial x_{i}\partial x_{j}}
\psi(1,\ldots,1,x_{i},1,\ldots,1,x_{j},\ldots,1)$ is integrable in $[0,1]^{2}$ and

\begin{align*}
&\int_{1}^{x_{j}}\int_{1}^{x_{i}}
\frac{\partial}{\partial y_{i}\partial y_{j}}
\psi(1,\ldots,1,y_{i},1,\ldots,1,y_{j},\ldots,1)
dy_{i}dy_{j}=\\
&\psi(1,\ldots,x_{i},\ldots,x_{j},\ldots,1)-
 \psi(1,\ldots,x_{i}   ,\ldots,1  ,\ldots,1)-\\
&\hspace{10em} -
\psi(1,\ldots,1    ,\ldots,x_{j},\ldots,1)+
 \psi(1,\ldots,1).
\end{align*}

\noindent Then, taking limit when $x\to 0$ to both sides of the previous formula we obtain that $\lim\limits_{x\to 0}\psi(1,\ldots,x_{i},\ldots,x_{j},\ldots,1)<\infty.$

\noindent If we continuous this process recursively, in the $(n-1)$ step then we obtain that $\lim\limits_{x\to 0}\psi(b)$ is finite if $b=(1,x_{2},\ldots x_{n})$, or $(x_{1},1,\ldots,x_{n})$, etc. Finally from (\ref{eq10}) we deduce (\ref{taylor7}).

	\smallskip
	
	Now let us make the following observation. If $r,p\in\N_{0}$
	
	\begin{center}
		\begin{equation}
		T_{i}^{r}x_{i}^{2p}=
		\begin{cases}
		\hfil 2^{r}r!                                  &\text{if}\:r=p \\
		\hfil 2^{r} \frac{p!}{(p-r-1)!}x_{i}^{2(p-r)}  &\text{if}\:r<p\\
		\hfil 0                                  &\text{if}\:r>p
		\end{cases}.
		\end{equation}
	\end{center}

	\noindent Let $m$ and $k$ be  multi-index such as $|m|=|k|=r$, then 
	\begin{center}
		\begin{equation} 
		T^{m}\{x^{2k}\}=
		\begin{cases}
		\hfil 2^{r}k!    &\text{if}\:m=k\\
		\hfil 0                       &\text{if}\:m\neq k
		\end{cases}.
		\end{equation}
	\end{center}
	
	\noindent Upon choosing $a_{2k_{r}}$ according to (\ref{taylor2}) and observing that
	
	\begin{align*}
	\lim\limits_{\underset{x_{i}>0}{x\to 0}} 
	T^{k_{r}}R_{2r}(x)
	&=
	\lim\limits_{\underset{x_{i}>0}{x\to 0}} 
	T^{k_{r}}\left\{
	x^{-\mu-1/2}\phi(x)-\sum_{j=0}^{r}
	\sum_{|k_{j}|=j} a_{k_{j}}x^{2k_{j}}\right\}=\\
	&=\lim\limits_{\underset{x_{i}>0}{x\to 0}} 
	T^{k_{r}}\{x^{-\mu-1/2}\phi(x)\}-
	a_{k_{r}}2^{r}k_{r}! =0,
	\end{align*}
	
	\noindent we obtain (\ref{taylor3}). Condition (iii) was already proved in \cite[Lemma 2.1]{Molina1}. Conversely, if conditions (i) and (ii) hold, then
	
	\[
	\sup_{x\in (0,1]^{n}}
	|(1+\|x\|^{2})T^{k}\{x^{-\mu-1/2}\phi(x)\}|<\infty.
	\]  
	
	\noindent From (\ref{Leibniz}) it can be deduce the formula
	$$T^{k}\{x^{-\mu-1/2}\}=
	x^{-\mu-1/2}\left\{
	\sum_{j=0}^{k}b_{k,j}\frac{D^{j}\phi}{x^{2k-j}}
	\right\}$$

	which implies jointly with conditions  (i) and (iii)  that 
	
	\[
	\sup_{x\in (1,\infty)^{n}}
	|(1+\|x\|^{2})T^{k}\{x^{-\mu-1/2}\phi(x)\}|<\infty.
	\] 
	
	Therefore $\gamma_{m,k}^{\mu}(\phi)$ are finite for all $m\in\N_{0}$ and  $k\in\N_{0}^{n}$ which completes the theorem. 
	
\end{proof}

\smallskip

Let $\mathcal{O}$ be the space of functions $\theta\in C^{\infty}(\Rn)$ with the property that for every $k\in\N_{0}^{n}$ there exists $n_{k}\in\Z$ and $C>0$ such that

$$|(1+\|x\|^{2})^{n_{k}}T^{k}\theta|<C \quad \forall x\in\Rn.$$

\bigskip

For the next Lemma, we will consider polynomials of $n$-variables of the form

\begin{equation*}
   P[x]=P[x_{1},\ldots,x_{n}]=
   \sum_{|\alpha|\leq N}a_{\alpha}x^{\alpha}\qquad\text{with}\;a_{\alpha}\in\R.
\end{equation*}

\begin{lemma}
  Let $P[x]$ and $Q[x]$ be polynomials of $n$-variables such that $Q[x]=\underset{|\alpha|\leq N}{\sum} b_{\alpha}x^{\alpha}\neq 0$  for all $x\in[0,\infty)^{n}$  and  all its coefficients have the same sign  then $\frac{P[x_{1}^{2},\ldots,x_{n}^{2}]}{Q[x_{1}^{2},\ldots,x_{n}^{2}]}\in \mathcal{O}$.
\end{lemma}
\begin{proof}
  Let us show that $P[x_{1}^{2},\ldots,x_{n}^{2}]\in \mathcal{O}$. We want to see that for all $k\in\N_{0}^{n}$ there exists $n_{k}\in\Z$ such that
\begin{equation}\label{O1}
  |(1+\|x\|^{2})^{n_{k}}T^{k}P[x_{1}^{2},\ldots,x_{n}^{2}]|<\infty .
\end{equation}

If $k=e_{i}$,

\begin{equation*}
\begin{aligned}
T^{e_{i}}P[x_{1}^{2},\ldots,x_{n}^{2}]
&= x_{i}^{-1}\frac{\partial}{\partial x_{i}}P[x_{1}^{2},\ldots,x_{n}^{2}]
= x_{i}^{-1}\frac{\partial}{\partial x_{i}}
\left(\sum_{|\zeta|\leq N'}a_{\zeta} 
x_{1}^{2\zeta_{1}}\ldots x_{n}^{2\zeta_{n}}\right)\\
&= x_{i}^{-1}\sum_{|\zeta|\leq N'}
2\zeta_{i}\,a_{\zeta}\,x_{1}^{2\zeta_{1}}\ldots x_{i}^{2\zeta_{i}-1}\ldots x_{n}^{2\zeta_{n}}
= \tilde{P}[x_{1}^{2},\ldots,x_{n}^{2}].\\
\end{aligned}
\end{equation*}

Any polynomial of the form $\underset{{|\beta|\leq N}}{\sum} c_{\beta} \: x_{1}^{2\beta_{1}}\ldots x_{n}^{2\beta_{n}}$ can be bounded in the following way

\begin{equation*}
\left|\sum_{|\beta|\leq N}c_{\beta}\: 
   x_{1}^{2\beta_{1}}\ldots x_{n}^{2\beta_{n}}\right|\leq
 \sum_{|\beta|\leq N} |c_{\beta}|
  |x_{1}^{2\beta_{1}}|\ldots |x_{n}^{2\beta_{n}}|
  <C(1+\|x\|^{2})^{|\gamma|},
 \end{equation*}

\noindent for suitables $C>0$ and a multi-index $\gamma$. So

$$|(1+\|x\|^{2})^{-|\gamma'|}\tilde{P}[x_{1}^{2},\ldots,x_{n}^{2}]|<C,$$
\noindent for some multi-index $\gamma'$. 

Now let us see  that $1/Q[x_{1}^{2},\ldots,x_{n}^{2}]$ is also in  $\mathcal{O}$. Let  $Q[x]=\underset{|\alpha|\leq N}{\sum} b_{\alpha}x_{1}^{\alpha_{1}}\ldots x_{n}^{\alpha_{n}}$ and without loss of generality we assume that $b_{\alpha}\geq 0$, $ \quad\forall\alpha:|\alpha|\leq N$, then 

\begin{align}
T^{e_{i}}(Q[x_{1}^{2},\ldots,x_{n}^{2}])^{-1} 
&= x_{i}^{-1}\frac{\partial}{\partial x_{i}}
   (Q[x_{1}^{2},\ldots,x_{n}^{2}])^{-1} \nonumber \\
&= x_{i}^{-1}(-1) (Q[x_{1}^{2},\ldots,x_{n}^{2}])^{-2}
   \frac{\partial }{\partial x_{i}} Q[x_{1}^{2},\ldots,x_{n}^{2}]\nonumber\\
&= (Q[x_{1}^{2},\ldots,x_{n}^{2}])^{-2}
   \tilde{Q}[x_{1}^{2},\ldots,x_{n}^{2}]\label{O3},
\end{align}

\noindent since $Q[x]$ does not have any zeros in $[0,\infty)^{n}$ then $b_{0}\ne 0$, so

\begin{equation*}
  Q[x_{1}^{2},\ldots,x_{n}^{2}]=b_{0}+\sum_{0<|\alpha|\leq N} b_{\alpha}x_{1}^{2\alpha_{1}}\ldots x_{n}^{2\alpha_{n}}\geq b_{0},
\end{equation*}

\noindent therefore
\begin{equation}\label{O2}
(Q[x_{1}^{2},\ldots,x_{n}^{2}])^{-2}\leq\frac{1}{b_{0}^{2}}<\infty.
\end{equation}

From (\ref{O3}) and (\ref{O2}), it follows (\ref{O1}) for $k=e_{i}$. The general case follows in a similar way.
\end{proof}

\section{Proofs of Liouville type theorem  in $\mathcal{H'}_{\mu}$}

The following is a representation theorem for distibutions  "supported in zero" in $\mathcal{H'}_{\mu}$.

\begin{theorem}\label{Liouville}
	Let $T\in\mathcal{H'}_{\mu}$ satisfying $(T, \phi)=0$ for all $\phi\in\mathcal{H}_{\mu}$ with $\supp(\phi)\subset \{x\in\Rn:\|x\|\geq a\}$ for some $a\in\R$,  $a>0$. Then there exist $N\in\N_{0}$ and scalars $c_{k}$, $|k|\leq N$ such that
		$$T=\sum_{|k|\leq N} c_{k} S^{k}\delta_{\mu},$$
	where $\delta_{\mu}$ is given by (\ref{T^{k}delta}) for $k=0$.
\end{theorem}

\begin{proof}The proof will follow directly from  \cite[Lemma 1.4.1]{Kesavan}  if we can show that  there exist $N_{0}$ such that if $\phi\in\mathcal{H}_{\mu}$ satisfies $(S^{k}\delta_{\mu}, \phi)=0$ for  $|k|\leq N_{0}$, then $(T, \phi)=0$.

Consider the family of seminorms $\{\lambda_{m,k}^{\mu}\}$ defined by (\ref{S2}) which generate the same topology in $\mathcal{H}_{\mu}$ as the family $\{\gamma_{m,k}^{\mu}\}$ (see Appendix A) and  let   

\begin{equation*}
\rho_{R}^{\mu}(\phi)=\sum_{\underset{|k|\leq R}{m\leq R}}\lambda_{m,k}^{\mu}(\phi).
\end{equation*}

This  family of seminorms result to be an increasing and equivalent to $\{\lambda_{m,k}^{\mu}\}$. So, given $T\in\mathcal{H'}_{\mu}$, there exist $c>0$ and $N\in\N_{0}$   such that 

\begin{equation*}
|(T, \phi)|\leq C \rho_{N}^{\mu}(\phi), \qquad \phi\in\mathcal{H}_{\mu}.
\end{equation*}\\

Now, let $\phi\in\mathcal{H}_{\mu}$ satisfying $(S^{k}\delta_{\mu}, \phi)=0$, for all $|k|\leq N_{0}$, where $N_{0}=2N$ then:

\begin{equation*}
\lim\limits_{\underset{x_{i}>0}{x\to 0}}
 x^{-\mu-1/2}S^{k}\phi(x)=0.
\end{equation*}

\noindent Given $\varepsilon>0 $ there exists $\eta_{k}>0$ such as $|x^{-\mu-1/2}S^{k}\phi(x)|<\varepsilon$, for all $x\in\Rn$ ,  $\|x\|< \eta_{k}$ for all $k$ such that $|k|<N_{0}$ .

\smallskip

Set $\eta=\underset{|k|\leq N_{0}}{\min}\{\eta_{k}\}$ and $\eta<1$, then

\begin{equation*}
|x^{-\mu-1/2}S^{k}\phi(x)|<\varepsilon \qquad\forall x\in\Rn : \|x\|< \eta.
\end{equation*}

Fix $\eta^{*}$  satisfying $0<\eta^{*}<\eta<1$ and define a smooth function $\psi$ on $\Rn$ by   $\psi(x)= 1$ for $\{x\in\Rn:\|x\|<\eta^{*} \}$ and   $\psi(x)= 0$ for $\{x\in\Rn:\|x\|\geq\eta\}$.

\smallskip

We claim that $\psi\in\mathcal{O}$. In fact, since $\psi\in C^{\infty}(\Rn)$ there exist $M_{k}>0$ such that $|T^{k}\psi(x)|\leq M_{k}$ then there  exist $n_{k}\in\N$ such that 
\[
|(1+\|x\|^{2})^{-n_{k}}T^{k}\psi(x)|<\infty.
\]

\medskip

Since $\supp((1-\psi)\phi)\subset \{x\in\Rn : \|x\|\geq \eta^{*}\}$, then for the hypothesis
 
\[
((1-\psi)T, \phi)=(T, (1-\psi)\phi)=0\qquad \forall\phi\in\mathcal{H}_{\mu}.
\]

\noindent From the above it follows that $T=\psi T$, then

\begin{align}
|(T,\phi)| &= |(\psi T,\phi)|=|(T, \psi\phi)|\leq C\rho_{N}^{\mu}(\psi\phi)=\nonumber\\
& = C \sum_{\underset{|k|\leq N}{m\leq N}}
\sup_{x\in\Rn}|(1+\|x\|^{2})^{m}x^{-\mu-1/2}S^{k}(\psi\phi)(x)|.\label{L5}
\end{align}

\noindent Since $\supp \psi \subset\{x\in\Rn : \|x\|\leq\eta\}$, then 

$$\sup_{x\in\Rn}|(1+\|x\|^{2})^{m}x^{-\mu-1/2}S^{k}(\psi\phi)(x) 
\leq$$ 
\begin{equation}\label{L6}
\sup_{\|x\|<\eta^{*}}|(1+\|x\|^{2})^{m}x^{-\mu-1/2}S^{k}\phi(x)|
+ \sup_{\eta^{*}\leq\|x\|<\eta}|(1+\|x\|^{2})^{m}x^{-\mu-1/2}S^{k}(\psi\phi)(x)|.
\end{equation}

\medskip

\noindent If we consider $\|x\|<\eta^{*}$, then

\begin{equation}\label{L7}
\sup_{\|x\|\leq\eta^{*}}|(1+\|x\|^{2})^{m}x^{-\mu-1/2}S^{k}\phi(x)|
\leq 2^{|m|}\varepsilon.
\end{equation}

\noindent Now we consider $\eta^{*}\leq\|x\|<\eta$. Applying (\ref{equiv2}) and (\ref{Leibniz}) we obtain that

\begin{align}
x^{-\mu-1/2} S^{k}(\psi\phi)(x)
&=\sum_{l=0}^{k}b_{l,k}x^{2l}T^{k+l}\{x^{-\mu-1/2}(\psi\phi)(x)\}\nonumber\\
&=\sum_{l=0}^{k}b_{l,k}x^{2l}
  \sum_{r=0}^{k+l}\binom{k+l}{r}T^{k+l-r}\psi(x)T^{r}\{x^{-\mu-1/2}\phi(x)\}.\label{L3}
\end{align}

\noindent Since $\psi\in C^{\infty}(\Rn)$ , there exist positive constants such that

\begin{equation}\label{L4}
|T^{k+l-r}\psi(x)|\leq M_{k,l,r},
\end{equation}

\noindent in  $\eta^{*}\leq\|x\|<\eta$. Accordingly to (\ref{L3}) and (\ref{L4}) we now have that

\begin{align}
|(1+\|x\|^{2})^{m}&x^{-\mu-1/2}S^{k}(\psi\phi)(x)|=\nonumber\\
&=\left|(1+\|x\|^{2})^{m}
\sum_{l=0}^{k}b_{l,k}x^{2l}
\sum_{r=0}^{k+l}
\binom{k+l}{r}T^{k+l-r}\psi(x)T^{r}\{x^{-\mu-1/2}\phi(x)\}\right|\nonumber\\
&\leq(1+\|x\|^{2})^{m}
\sum_{l=0}^{k}\sum_{r=0}^{k+l}
|b_{l,k}|\binom{k+l}{r} 
|T^{k+l-r}\psi(x)|\,\left|x^{2l}T^{r}\{x^{-\mu-1/2}\phi(x)\}\right|\nonumber\\
&\leq(1+\|x\|^{2})^{m}
\sum_{l=0}^{k}\sum_{r=0}^{k+l}
|b_{l,k}|\binom{k+l}{r}
M_{k,l,r}\,\left|x^{2l}T^{r}\{x^{-\mu-1/2}\phi(x)\}\right|\nonumber\\
&=
\sum_{l=0}^{k}\sum_{r=0}^{k+l} 
M^{*}_{k,l,r}\,(1+\|x\|^{2})^{m}x^{2l}|T^{r}\{x^{-\mu-1/2}\phi(x)\}|\nonumber\\
&\leq
\sum_{l=0}^{k}\sum_{r=0}^{k+l} 
M^{*}_{k,l,r}\,(1+\|x\|^{2})^{m+l}|T^{r}\{x^{-\mu-1/2}\phi(x)\}|\nonumber\\
&\leq
\sum_{l=0}^{k}\sum_{r=0}^{k+l} 
B_{k,l,r}\sup_{x\in\Rn}
|(1+\|x\|^{2})^{m+l}x^{-\mu-1/2}S^{r}\phi(x)|.\label{L8}
\end{align}

\noindent Since $|r|\leq |2k|\leq 2N=N_{0}$ then \begin{equation}\label{L9}
|(1+\|x\|^{2})^{m+l}x^{-\mu-1/2}S^{r}\phi(x)|\leq
2^{|m+l|}|x^{-\mu-1/2}S^{r}\phi(x)|\leq 2^{|m+l|} \varepsilon. 
\end{equation}

\noindent From (\ref{L5}), (\ref{L6}), (\ref{L7}), (\ref{L8}) and (\ref{L9}) then:
\begin{align*}
|(T,\phi)|&\leq
C \sum_{\underset{|k|\leq N}{m\leq N}}
\sup_{x\in\Rn}|(1+\|x\|^{2})^{m}x^{-\mu-1/2}S^{k}(\psi\phi)(x)|\leq\\
&\leq C\sum_{\underset{|k|\leq N}{m\leq N}}
\left(
2^{|m|}\varepsilon+
\sum_{l=0}^{k}\sum_{r=0}^{k+l} B_{k,l,r}2^{|m+l|}\varepsilon\right)=C'\varepsilon
\end{align*}

\noindent with $C'=C\sum_{\underset{|k|\leq N}{m\leq N}}
\left(
2^{|m|}+
\sum_{l=0}^{k}\sum_{r=0}^{k+l} B_{k,l,r}2^{|m+l|}\right)$. Hence $(T,\phi)=0$ since $\varepsilon>0$ was arbitrarily chosen.

\end{proof}

\begin{lemma}
  Let $\psi\in C^{\infty}(\Rn)$ such that  $\psi(x)=1$ if  $x_{1}+\ldots+x_{n}\geq a^{2}\:$,  $\psi(x)=0$ if $x_{1}+\ldots+x_{n}\leq b^{2}$ with $0<b^{2}\leq a^{2}$ and $0\leq\psi\leq1$. And let $P[x]=\underset{|\alpha|\leq N}{\sum} a_{\alpha}x^{\alpha}\neq 0$ for all $x\in\R^{n}-\{0\}$ and all its coefficients have the same sign  , therefore

  $$P[x_{1}^{2},\ldots,x_{n}^{2}]^{-1}\psi(x_{1}^{2},\ldots,x_{n}^{2})\in\mathcal{O}.$$
\end{lemma}

\begin{proof}

Let $P[x_{1},\ldots, x_{n}]=
\underset{|\alpha|\leq N}{\sum}
a_{\alpha}x_{1}^{\alpha_{1}}\ldots x_{n}^{\alpha_{n}}$.

The aim of this proof is to verify that for all $k\in\N_{0}^{n}$ there exists $n_{k}\in\Z$ such that 

\begin{equation*}
|(1+\|x\|^{2})^{n_{k}}T^{k}\{P[x_{1}^{2},\ldots,x_{n}^{2}]^{-1}\psi(x_{1}^{2},\ldots,x_{n}^{2})\}|\leq C \quad \forall x\in\Rn.
\end{equation*}

\smallskip

\noindent For $b\leq \|x\|\leq a$ it turns out that

\begin{eqnarray}
&& T^{e_{i}} \{P[x_{1}^{2},\ldots,x_{n}^{2}]^{-1}
              \psi(x_{1}^{2},\ldots,x_{n}^{2})\}
 =x_{i}^{-1}\frac{\partial}{\partial x_{i}}    \{P[x_{1}^{2},\ldots,x_{n}^{2}]^{-1}
\psi(x_{1}^{2},\ldots,x_{n}^{2})\}=\nonumber\\
&& =P[x_{1}^{2},\ldots,x_{n}^{2}]^{-2}
               \tilde{P}[x_{1}^{2},\ldots,x_{n}^{2}]
               \psi(x_{1}^{2},\ldots,x_{n}^{2})+\nonumber\\
&&\hspace{16em}
  +2P[x_{1}^{2},\ldots,x_{n}^{2}]^{-1}\frac{\partial \psi}{\partial x_{i}}(x_{1}^{2},\ldots,x_{n}^{2}).\label{L1}
\end{eqnarray}

Since all the functions involved, $\psi$  and its derivatives are all continuous in $b\leq\|x\|\leq a$,  it is clear that (\ref{L1}) is bounded. On the other hand, if $\|x\|\geq a$, since $\psi(x)=1$ then

 \begin{eqnarray*}
T^{e_{i}}\{P[x_{1}^{2},\ldots,x_{n}^{2}]^{-1}\}
  &=& x_{i}^{-1}\frac{\partial}{\partial x_{i}} \{P[x_{1}^{2},\ldots,x_{n}^{2}]^{-1}\}\\
 &=& P[x_{1}^{2},\ldots,x_{n}^{2}]^{-2}
                   \tilde{P}[x_{1}^{2},\ldots,x_{n}^{2}]
 \end{eqnarray*}

We already shown that $\tilde{P}$ is in $\mathcal{O}$, so, there exist $r\in\Z$ such that $|\tilde{P}[x_{1}^{2},\ldots,x_{n}^{2}]|\leq C(1+\|x\|^{2})^{r}$. Without loss of generality suppose that all $a_{\alpha}$ are positives and let us first consider  $a_{0}\neq 0$, then $P[x_{1}^{2},\ldots,x_{n}^{2}]^{-2}$ is bounded as in (\ref{O3}).

\medskip
If now we consider $a_{0}=0$, since $P[x_{1}^{2},\ldots,x_{n}^{2}]> 0$ for $(x_{1},\ldots,x_{n})\neq (0,\ldots,0)$ then $P$  must attain  a minimum  in $S^{n-1}$. Let  $\delta$ be such that

\begin{equation}\label{L10}
\delta< P\left[\frac{x_{1}^{2}}{\|x\|^{2}},\ldots,
\frac{x_{n}^{2}}{\|x\|^{2}}
\right]
=\sum_{1\leq|\alpha|\leq N}a_{\alpha}
\frac{x_{1}^{2\alpha_{1}}\ldots x_{n}^{2\alpha_{n}}}{\|x\|^{2|\alpha|}}.
\end{equation}

Since $\|x\|\geq a$ and $|\alpha|\geq 1$ then
\begin{equation}\label{L11}
\|x\|^{2|\alpha|}>a^{2|\alpha|}
\end{equation}

From (\ref{L10}) and (\ref{L11}) we obtain that 

\begin{equation}
\delta <C\sum_{1\leq|\alpha|\leq N}a_{\alpha}x_{1}^{2\alpha_{1}}\ldots x_{n}^{2\alpha_{n}}
\end{equation}

\noindent with  $C=
\underset{1\leq|\alpha|\leq N}{\max} a^{-2|\alpha|}$, then

$$P[x_{1}^{2},\ldots,x_{n}^{2}]^{-2}\leq 
C^{2}\delta^{-2}.$$

Then, 

\begin{equation}\label{L2}
\sup_{\|x\|\geq a}
|T^{e_{i}}\{P[x_{1}^{2},\ldots,x_{n}^{2}]^{-1}|\leq
C'(1+\|x\|^{2})^{r}.
\end{equation} 

From equations (\ref{L1}) and (\ref{L2}) the Lemma follows for $k=e_{i}$. The general case follows in a similar way.

\end{proof}

Now we are ready for the proof of the Theorem \ref{GeneralLiouville}.
\begin{proof}[Proof of Theorem \ref{GeneralLiouville}]

 If $L(f)=0$ this means that

 $$\sum_{|\alpha|\leq N}(-1)^{|\alpha|}a_{\alpha}S^{\alpha}f=0.$$

\noindent Since $h_{\mu}(S_{\mu_{i}}f)=-y_{i}^{2}h_{\mu}f$ (see \cite{Molina2}),  applying Hankel transform to both sides, we have

\begin{eqnarray}
h_{\mu}\left(\sum_{|\alpha|\leq N}(-1)^{|\alpha|}a_{\alpha}S^{\alpha}f \right)&=&   \sum_{|\alpha|\leq N}(-1)^{|\alpha|}a_{\alpha}(-1)^{|\alpha|} y_{1}^{2\alpha_{1}}\ldots y_{n}^{2\alpha_{n}}h_{\mu} f =\nonumber\\
&=&P[y_{1}^{2},\ldots, y_{n}^{2}]h_{\mu} f=0 .\label{GL1}
\end{eqnarray}

 Let $\psi$ being as in the previous Lemma. Then $[P[y_{1}^{2},\ldots, y_{n}^{2}]]^{-1}\psi(y_{1}^{2},\ldots, y_{n}^{2})\in\mathcal{O}$. Then multiplying  in (\ref{GL1}) we obtain that

\begin{equation}\label{GL2}
\psi(y_{1}^{2},\ldots, y_{n}^{2}).h_{\mu} f=0.
\end{equation}

\noindent Let $\phi\in\mathcal{H}_{\mu}$ with $\supp \phi\subset \{x\in\Rn :\|x\|\geq a\}$ and let us see that

$$(h_{\mu} f, \phi)=0. $$

\noindent Since  $\psi(x_{1}^{2},\ldots,x_{n}^{2}).\phi(x_{1},\ldots.x_{n})=\phi(x_{1},\ldots.x_{n})$ in $\Rn$, then

\begin{equation}\label{GL3}
(h_{\mu} f, \phi)=
(h_{\mu} f, \psi\phi)=(\psi h_{\mu} f, \phi)=0,
\end{equation}

\noindent where we have used (\ref{GL2}). Consequently $h_{\mu} f$ is zero for all $\phi$ such that $\supp \phi\subset \{x\in\Rn :\|x\|\geq a\}$. For Theorem \ref{Liouville} there exist $N_{1}\in\N_{0}$ and constants $c_{k}$, $|k|\leq N_{1}$ such that

\begin{equation}\label{GL4}
h_{\mu} f=\sum_{|k|\leq N_{1}}c_{k}S^{k}\delta_{\mu}.
\end{equation}

Therefore, applying the Hankel transform $h_{\mu}$ to both sides of (\ref{GL4}) and since $h_{\mu}=(h_{\mu})^{-1}$ we obtain that

\begin{eqnarray*}
f &=& h_{\mu}(h_{\mu} f)=\sum_{|k|\leq N_{1}}    c_{k}\nH_{\mu}(S^{k}\delta_{\mu})= \\
  &=& \sum_{|k|\leq N_{1}} c_{k}(-1)^{|k|}y_{1}^{2k_{1}}\ldots y_{1}^{2k_{n}}h_{\mu} \delta_{\mu} \\
   &=& \sum_{|k|\leq N_{1}}c_{k}(-1)^{|k|}y_{1}^{2k_{1}}\ldots y_{1}^{2k_{n}} y^{\mu+1/2},
\end{eqnarray*}
 which completes the proof.

\end{proof}

\section{Another proof of Theorem \ref{GeneralLiouville}}

We establish  a different representation theorem from the one proved in the previous section.

\begin{theorem}\label{Estructura2}
	Let $f\in\mathcal{H'}_{\mu}$ satisfying $(f, \phi)=0$ for all $\phi\in\mathcal{H}_{\mu}$ with $\supp(\phi)\subset \{x\in\Rn:\|x\|\geq a\}$ for some $a\in\R$,  $a>0$. Then there exist $N\in\N_{0}$ and scalars $c_{k}$, $|k|\leq N$ such that
	$$f=\sum_{|k|\leq N} c_{k} T^{k}\delta_{\mu},$$
	where $T^{k}\delta_{\mu}$ given by (\ref{T^{k}delta}).
\end{theorem}
\begin{proof}

	Let  $f\in\mathcal{H'}_{\mu}$, such that $f$ verifies the hypothesis of the theorem and  $c>0$, $N\in\N_{0}$ such that 
	
	\begin{equation}\label{E2}
	|(f,\phi)|\leq C \sum_{\underset{|k|\leq N}{m\leq N}}\gamma_{m,k}^{\mu}(\phi),   \qquad \phi\in\mathcal{H}_{\mu}.
	\end{equation}

	By the Taylor formula and (\ref{T^{k}delta}), if  $\phi\in\mathcal{H}_{\mu}$ 
	
	\begin{equation}\label{E3}
	\phi(x)=\frac{x^{\mu+1/2}}{C_{\mu}}
	\left\{
	(\delta_{\mu},\phi)+
	\sum_{|k_{1}|=1} (T^{k_{1}}\delta_{\mu},\phi)\frac{x^{2k_{1}}}{2 k_{1}!}
	+\ldots+
	\sum_{|k_{N}|=N}
	(T^{k_{N}}\delta_{\mu},\phi)\frac{x^{2k_{N}}}{2^{N} k_{N}!}+
	C_{\mu}R_{2N}(x)
	\right\}
	\end{equation}

\noindent 	where the remain term satisfies  
	$\lim\limits_{\underset{x_{i}>0}{x\to 0}}T^{k}R_{2N}(x)=0$ for all $k$ multi-index such that $|k|\leq N$. Then, given $\varepsilon>0$ there exist $\eta_{k}>0$ such that 
	$|T^{k}R_{2N}(x)|<\varepsilon  \quad$ for $x\in\Rn:\|x\|<\eta_{k}$. Set $\eta=\min_{|k|\leq N}\{\eta_{k}\}$ and $\eta<1$, then

	$$|T^{k}R_{2N}(x)|<\varepsilon  \quad\forall x\in\Rn:\|x\|<\eta, \quad |k|\leq N.$$
	
	Let $a\in\R$ such that $0<a<\eta$ and  define  $\psi$ a smooth function on $\Rn$ by $\psi(x)=1$ for $\{x\in\Rn:\|x\|< a/2\}$ and  $\psi(x)=0$ for $\{x\in\Rn:\|x\|\geq a\}$ and therefore $(f,(1-\psi(x))\phi(x)=0$ for any $\phi\in\mathcal{H}_{\mu}$. Hence 
	
	\begin{equation}\label{E1}
	(f,\phi)=(f,\psi\phi).
	\end{equation}

	Therefore


	\begin{equation}\label{E4}
	(f,\phi)=\sum_{|k|\leq N}
	c_{k}\,(T^{k}\delta_{\mu},\phi)+
	(f,x^{\mu+1/2}\psi(x)R_{2N}(x))
	\end{equation}
	
	where
	\begin{equation}\label{E5}
	c_{k}=\frac{1}{C_{\mu}2^{|k|}k!}
	(f, x^{\mu+1/2}x^{2k}\psi(x)).
	\end{equation} 
	
	Applying the estimate (\ref{E2}) to $x^{\mu+1/2}\psi(x)R_{2N}(x)$, we get 
	
	\[
	|(f,x^{\mu+1/2}\psi(x)R_{2N}(x))|\leq C \sum_{\underset{|k|\leq N}{m\leq N}}\gamma_{m,k}^{\mu}(x^{\mu+1/2}\psi(x)R_{2N}(x))
	\]

	\begin{align*}
	&|(1+\|x\|^{2})^{m}T^{k}\{x^{-\mu-1/2}x^{\mu+1/2}\psi(x)R_{2N}(x)\}|
	\leq\\
	&\sup_{\|x\|<\frac{a}{2}}      
	|(1+\|x\|^{2})^{m} T^{k}\{ R_{2N}(x)\}|  +
	\sup_{\frac{a}{2}\leq\|x\|<a}       
	|(1+\|x\|^{2})^{m} T^{k}\{\psi(x) R_{2N}(x)\}|\leq\\
	& \sup_{\|x\|<\frac{a}{2}}      
	|(1+\|x\|^{2})^{m} T^{k}\{ R_{2N}(x)\}|  +
	\sup_{\frac{a}{2}\leq\|x\|<a} 
	|(1+\|x\|^{2})^{m} \sum_{j=0}^{k} T^{k-j}\psi(x) T^{j}R_{2N}(x)|\leq\\
	&\sup_{\|x\|<\frac{a}{2}}      
	|(1+\|x\|^{2})^{m} T^{k}\{ R_{2N}(x)\}|  +
	\sup_{\frac{a}{2}\leq\|x\|<a} 
	\sum_{j=0}^{k} M_{j,k}|(1+\|x\|^{2})^{m} T^{j}R_{2N}(x)|
	\end{align*}

	For $\|x\|<\eta$

	\[
	|(f,x^{\mu+1/2}\psi(x)R_{2N}(x))|\leq C 
	\sum_{\underset{|k|\leq N}{m\leq N}}
	2^{m}\left(1+\sum_{j=0}^{k} M_{j,k}
	\right)\varepsilon =
	C'\varepsilon
	\]

	Thus $(f,x^{\mu+1/2}\psi(x)R_{2N}(x))=0$ since $\varepsilon$ was arbitrarily chosen. Therefore
	
	\[ (f,\phi)=\sum_{|k|\leq N} c_{k}\,(T^{k}\delta_{\mu},\phi) \]

\end{proof}

Now we can sketch a different proof for Theorem \ref{GeneralLiouville}.

\medskip
\begin{proof}[Another proof of Theorem \ref{GeneralLiouville}]
	
If $L(f)=0$, then we obtain as in (\ref{GL3}) that $h_{\mu}f$ is zero for all $\phi$ such that $\supp\phi\subset \{x\in\Rn :\|x\|\geq a\}$ with $a>0$, $a\in\R$. Then, since Theorem \ref{Estructura2} holds, there exist $N_{2}\in\N_{0}$ and constants $c_{k}$, $|k|\leq N_{2}$ such that 

\begin{equation}\label{GL5}
h_{\mu} f=\sum_{|k|\leq N_{2}}c_{k}T^{k}\delta_{\mu}.
\end{equation}

Therefore, applying the Hankel transform $h_{\mu}$ to both sides of (\ref{GL5}) and since $h_{\mu}=(h_{\mu})^{-1}$ we obtain that

\begin{eqnarray*}
	f = h_{\mu}(h_{\mu} f)
	&=&\sum_{|k|\leq N_{2}}    c_{k}h_{\mu}(T^{k}\delta_{\mu})= \\
	&=& \sum_{|k|\leq N_{2}} c_{k}M_{k}^{\mu}y^{2k}y^{\mu+1/2}\\
	&=& \sum_{|k|\leq N_{2}} c_{k}M_{k}^{\mu}
	    y_{1}^{2k_{1}}\ldots y_{n}^{2k_{n}}y^{\mu+1/2},\\
\end{eqnarray*}
where we have used (\ref{HankelT^{k}}). The proof is this complete. 

\end{proof}

\appendix
\section{Equivalence of the seminorms $\gamma_{m,k}^{\mu}$ and $\lambda_{m,k}^{\mu}$}

The main result of this paper needs of the existence of another family of seminorms, different from the family $\gamma_{m,k}^{\mu}$, which is defined as

\begin{equation}\label{S2}
\lambda_{m,k}^{\mu}(\phi)=\sup_{x\in\Rn}|(1+\|x\|^2)^{m}x^{-\mu-1/2}S^{k}\phi(x)|, \qquad \phi\in\mathcal{H}_{\mu}.
\end{equation}

The construction of the family $\{\lambda_{m,k}^{\mu}\}_{m\in\N_{0},k\in\N_{0}^{n}}$ was motivated by the work of \cite{Betancor1} , \cite{Sanchez} and \cite{Zemanian1}. This multinorm is important because generates on $\mathcal{H}_\mu$ the same topology as the family $\{\gamma_{m,k}^{\mu}\}$.

\begin{remark}
	Let $k$ be a multi-index, the following equality is valid
	\begin{equation}\label{equiv1}
	x^{-\mu-1/2} S^{k}\phi(x)=
	\sum_{l=0}^{k}b_{l,k}x^{2l}T^{k+l}\{x^{-\mu-1/2}\phi(x)\}.
	\end{equation}
\end{remark}

\noindent This formula can be derived from the equation 

\begin{equation}\label{equiv2}
x_{i}^{-\mu_{i}-1/2} S_{\mu_{i}}^{k_{i}}\phi(x)=\sum_{l=0}^{k_{i}}b_{l,k_{i}}x_{i}^{2l}T_{i}^{k_{i}+l}\{x_{i}^{-\mu_{i}-1/2}\phi(x)\},
\end{equation}

\noindent where the constants $b_{j,k_{i}}$, $j=0,\ldots,k_{i}$,  are suitable real constants, only depending on $\mu_{i}$. The formula  (\ref{equiv2}) is due to Koh and Zemanian \cite[(9)]{Zemanian1} and is valid for every $k_{i}\in\N_{0}$. Indeed, if $k\in\N_{0}^{n}$, $k=(k_{1},\ldots,k_{n})$ then, 

 \begin{align*}
\left(x_{i}^{-\mu_{i}-1/2} S_{\mu_{i}}^{k_{i}}\right)&
 \left(x_{j}^{-\mu_{j}-1/2} S_{\mu_{j}}^{k_{j}}\right)\phi(x)=\\
&=\sum_{l_{i}=0}^{k_{i}}b_{l_{i},k_{i}}x_{i}^{2l_{i}}T_{i}^{k_{i}+l_{i}}\{x_{i}^{-\mu_{i}-1/2}
\sum_{l_{j}=0}^{k_{j}}b_{l_{j},k_{j}}x_{j}^{2l_{j}}T_{j}^{k_{j}+l_{j}}\{x_{j}^{-\mu_{j}-1/2}\phi(x)\}\}\\
&=\sum_{l_{i}=0}^{k_{i}}
\sum_{l_{j}=0}^{k_{j}}
b_{l_{i},k_{i}}b_{l_{j},k_{j}}
x_{i}^{2l_{i}}x_{j}^{2l_{j}}
T_{i}^{k_{i}+l_{i}}
T_{j}^{k_{j}+l_{j}}
\{x_{i}^{-\mu_{i}-1/2} x_{j}^{-\mu_{j}-1/2}\phi(x)\}.
\end{align*}

\noindent Repeating this process we obtain that 

\begin{align*}
& x^{-\mu-1/2} S^{k}\phi(x)=
  \left(x_{1}^{-\mu_{1}-1/2} S_{\mu_{1}}^{k_{1}}\right)\ldots
  \left(x_{n}^{-\mu_{n}-1/2} S_{\mu_{n}}^{k_{n}}\right)\phi(x)=\\
& \sum_{l_{1}=0}^{k_{1}}\ldots\sum_{l_{n}=0}^{k_{n}}
        b_{l_{1},k_{1}} \ldots b_{l_{n},k_{n}}
         x_{1}^{2l_{1}} \ldots \: x_{n}^{2l_{n}}
    T_{1}^{k_{1}+l_{1}} \ldots \: T_{n}^{k_{n}+l_{n}}
 \{x_{1}^{-\mu_{1}-1/2 }\ldots \: x_{n}^{-\mu_{n}-1/2}\phi(x)\}=\\
& \sum_{l=0}^{k}b_{l,k}x^{2l}T^{k+l}\{x^{-\mu-1/2}\phi(x)\}.
\end{align*}

On the other hand, from \cite[Propositions IV.2.2 and IV.2.4]{Sanchez}  we have that for all $k_{i}\in\N_{0}$, $i=1,\ldots,n$

\begin{equation*}\label{equiv3}
|T_{i}^{k_{i}}\{x_{i}^{-\mu_{i}-1/2}\phi(x)\}|\leq C_{i} \sup_{x_{i}\in(0,\infty)}
|x_{i}^{-\mu_{i}-1/2}S_{\mu_{i}}^{k_{i}}\phi(x)|.
\end{equation*}

So, we can generalize this inequality and obtain the following result

\begin{remark}
	Let $k$ be a multi-index, the following inequality is valid
	\begin{equation}\label{equiv4}
	|T^{k}\{x^{-\mu-1/2}\phi(x)\}|\leq C \sup_{x\in\Rn}
	|x^{-\mu-1/2}S^{k}\phi(x)|.
	\end{equation}
\end{remark}

Set $i,j\in\{1,\ldots,n\}$, $i\neq j$ and computing 
\begin{align*}
 |T_{i}^{k_{i}}T_{j}^{k_{j}}&\{x^{-\mu-1/2}\phi(x)\}|=
  |T_{i}^{k_{i}}\{x_{i}^{-\mu_{i}-1/2}
   T_{j}^{k_{j}}\{x_{j}^{-\mu_{j}-1/2}
  x^{-\mu-1/2+(\mu_{i}e_{i}+1/2)+(\mu_{j}e_{j}+1/2)}\phi(x)\}\}|\leq\\
&\leq C_{i} \sup_{x_{i}\in(0,\infty)}
  |x_{i}^{-\mu_{i}-1/2}S_{\mu_{i}}^{k_{i}}
  \{T_{j}^{k_{j}}\{x_{j}^{-\mu_{j}-1/2}
  x^{-\mu-1/2+(\mu_{i}e_{i}+1/2)+(\mu_{j}e_{j}+1/2)}\phi(x)\}\}|\\
& = C_{i} \sup_{x_{i}\in(0,\infty)}
 |T_{j}^{k_{j}}\{x_{j}^{-\mu_{j}-1/2}
 x^{-\mu-1/2+(\mu_{i}e_{i}+1/2)+(\mu_{j}e_{j}+1/2)}
 \left(x_{i}^{-\mu_{i}-1/2}S_{\mu_{i}}^{k_{i}}\right)\phi(x)\}|\\
& = C_{i} \sup_{x_{i}\in(0,\infty)}
 |T_{j}^{k_{j}}\{x_{j}^{-\mu_{j}-1/2}
 \left(x^{-\mu-1/2+(\mu_{j}e_{j}+1/2)}S_{\mu_{i}}^{k_{i}}\right)
 \phi(x)\}|\\
& \leq C_{i}C_{j}
 \sup_{x_{j},x_{i}\in(0,\infty)}
 |x_{j}^{-\mu_{j}-1/2}S_{\mu_{j}}^{k_{j}}\{
 x^{-\mu-1/2+(\mu_{j}e_{j}+1/2)}S_{\mu_{i}}^{k_{i}}
 \phi(x)\}|\\
& = C_{i}C_{j}
 \sup_{x_{j},x_{i}\in(0,\infty)}
 |x^{-\mu-1/2+(\mu_{j}e_{j}+1/2)}x_{j}^{-\mu_{j}-1/2}
 S_{\mu_{i}}^{k_{i}}
 S_{\mu_{j}}^{k_{j}}\phi(x)|\\
& = C_{i}C_{j}
 \sup_{x_{j},x_{i}\in(0,\infty)}|x^{-\mu-1/2}
 S_{\mu_{i}}^{k_{i}}
 S_{\mu_{j}}^{k_{j}}\phi(x)|.\\
 \end{align*}
\noindent The general case follows from an inductive argument.

\medskip

From (\ref{equiv1}) and (\ref{equiv4}) we obtain that the families of seminorms $\gamma_{m,k}^{\mu}$ and $\lambda_{m,k}^{\mu}$ are equivalents. 

\[
\left|(1+\|x\|^{2})^{m}T^{k}\{x^{-\mu-1/2}\phi(x)\}\right|\leq
 C \sup_{x\in\Rn}
|(1+\|x\|^{2})^{m}x^{-\mu-1/2}S^{k}\phi(x)|= C\lambda_{m,k}^{\mu}(\phi)
\]

\noindent therefore $\gamma_{m,k}^{\mu}\phi(x)\leq C\lambda_{m,k}^{\mu}\phi(x)$. On the other hand,  (\ref{equiv1}) imply that

\begin{align*}
\left|(1+\|x\|^{2})^{m}x^{-\mu-1/2} S^{k}\phi(x)\right|
&\leq\sum_{l=0}^{k}|b_{l,k}|\:
|(1+\|x\|^{2})^{m}x^{2l}T^{k+l}\{x^{-\mu-1/2}\phi(x)\}|\\
&\leq\sum_{l=0}^{k}|b_{l,k}|\:
|(1+\|x\|^{2})^{m+|l|}T^{k+l}\{x^{-\mu-1/2}\phi(x)\}|\\
&\leq\sum_{l=0}^{k}|b_{l,k}|\:
\gamma_{m+|l|,k+l}^{\mu}\phi(x),
\end{align*}

\noindent which leads to $\lambda_{m,k}^{\mu}\phi(x)\leq\sum_{l=0}^{k}|b_{l,k}|\:
\gamma_{m+|l|,k+l}^{\mu}\phi(x)$.

\end{document}